\documentclass[11pt,a4paper]{amsart}
\usepackage{amssymb}
\usepackage{mathabx}
\usepackage{mathrsfs}
\usepackage{syntonly}
\usepackage{amsmath}
\usepackage{amsthm}
\usepackage{amsfonts}
\usepackage{amssymb}
\usepackage{latexsym}
\usepackage{amscd,amssymb,amsopn,amsmath,amsthm,graphics,amsfonts,mathrsfs,accents,enumerate,verbatim,calc}
\usepackage[dvips]{graphicx}
\usepackage[colorlinks=true,linkcolor=red,citecolor=blue]{hyperref}
\usepackage[all]{xy}

\date{}
\allowdisplaybreaks[4] \footskip=15pt
\renewcommand{\uppercasenonmath}[1]{}

\numberwithin{equation}{section} \theoremstyle{plain}
\newtheorem{lem}{Lemma}[section]
\newtheorem{cor}[lem]{Corollary}

\newtheorem{thm}[lem]{Theorem}

\newtheorem{definition}[lem]{Definition}
\newtheorem{Ex}[lem]{Example}
\newtheorem{Quest}[lem]{Question}
\newtheorem{Property}[lem]{Property}
\newtheorem{Properties}[lem]{Properties}
\newtheorem{Subprops}{}[lem]
\newtheorem{Para}[lem]{}

\newtheorem{remark}[lem]{Remark}
\newtheorem{rem}[lem]{Remark}

\newenvironment{rmk}{\begin{rem}\rm}{\end{rem}}
\newenvironment{df}{\begin{definition}\rm}{\end{definition}}

\newtheorem*{ack*}{ACKNOWLEDGEMENTS}

\newcommand{\pf}{\noindent\begin {proof}}
\newcommand{\epf}{\end{proof}}

\newcommand{\X}{\mathcal{X}}
\newcommand{\W}{\mathcal{W}}
\newcommand{\U}{\mathcal{U}}
\newcommand{\V}{\mathcal{V}}

\newcommand{\C}{\mathcal{C}}

\newcommand{\E}{\mathbb{E}}

\newcommand{\ra}{\rightarrow}

\newcommand{\D}{\mathrm{Defl}}
\newcommand{\I}{\mathrm{Infl}}
\newcommand{\Zh}{\mathrm{Cone}}
\newcommand{\YZ}{\mathrm{CoCone}}
\newcommand{\We}{\W_{\widetilde{\mathcal{L}},\widetilde{\mathcal{R}}}}
\begin{document}
\begin{center}
{\large  \bf  Admissible weak factorization systems on extriangulated categories}

\vspace{0.5cm}  Yajun Ma, Hanyang You, Dongdong Zhang$\footnote{Corresponding author. Panyue Zhou is supported by the National Natural Science Foundation of China (Grant No. 12371034) and by the Hunan Provincial Natural Science Foundation of China (Grant No. 2023JJ30008).}$ and Panyue Zhou

\end{center}

\bigskip
\centerline { \bf  Abstract}
\medskip

\leftskip10truemm \rightskip10truemm \noindent\hspace{1em}  Extriangulated categories, introduced by Nakaoka and Palu, serve as a simultaneous generalization of exact and triangulated categories. In this paper, we first introduce the concept of admissible weak factorization systems and establish a bijection between cotorsion pairs and admissible weak factorization systems on extriangulated categories. Consequently, we give the equivalences between hereditary cotorsion pairs and compatible cotorsion pairs via admissible weak factorization systems under certain conditions in extriangulated categories, thereby generalizing a result by Di, Li, and Liang.\\[2mm]
{\bf Keywords:} extriangulated category; admissible weak factorization system;  cotorsion  pair\\
{\bf 2020 Mathematics Subject Classification:} { 18G80; 18E10; 18A32}

\leftskip0truemm \rightskip0truemm
\section { \bf Introduction}

The notion of extriangulated categories, defined by $\E$-triangles satisfying specific axioms, was introduced by Nakaoka and Palu in \cite{NP} as a simultaneous generalization of exact and triangulated categories. More precisely, triangulated categories and extension-closed subcategories of triangulated categories are examples of extriangulated categories. Additionally, there exist other examples of extriangulated categories that are neither exact nor triangulated (see \cite{NP, ZZ, HZZ}).

A model structure on a category $\C$ is a triple of three classes of morphisms respectively called \emph{cofibrations}, \emph{fibrations}, and \emph{weak equivalences}, satisfying Retract axiom, Lifting axiom, Factorization axiom and Two out of three axiom. For more details, see \cite{HM,QH}. Given an additive category $\C$ with a model structure, Quillen's homotopy category (i.e., the localization of $\C$ with respect to the weak equivalences) is a pretriangulated category in the sense of \cite{BR}. However, it is not a triangulated category in general.

A triple $(\U, \mathcal{W}, \V)$ of classes of objects in an exact category $\C$ is  called a \emph{Hovey triple}  if both $(\U, \mathcal{W}\bigcap\V)$ and $(\U\bigcap \mathcal{W}, \V)$ are complete cotorsion pairs, and $\mathcal{W}$ is \emph{thick} (i.e., $\mathcal{W}$ is closed under direct summands and if two out of three terms in an admissible exact sequence are in $\mathcal{W}$, then so is the third one).

Nakaoka and Palu established a bijective correspondence between admissible model structures and Hovey triples \cite{NP}, unifying Hovey's work on abelian categories \cite{HCc}, Gillespie's work on weakly idempotent complete exact categories \cite{Gillespie}, and Yang's work on triangulated categories \cite{Yang}. These correspondences offer a method for constructing model structures from Hovey triples, which are simpler to work with. For example, Gillespie provided a convenient way to construct Hovey triples using compatible and complete hereditary cotorsion pairs (see \cite{G}).


\vspace{2mm}
 {\bf Theorem (Gillespie)} \cite[Theorem 1.1]{G} Let $(\U,\widetilde{\V})$ and $(\widetilde{\U},\V)$ be compatible and complete hereditary cotorsion pairs in an abelian category $\mathcal{C}$. Then there is a subcategory $\mathcal{W}$ such that $(\U, \mathcal{W}, \V)$ forms a Hovey triple.

Zhou  generalized the above theorem to extriangulated categories, see \cite{Z}.
\vspace{2mm}

 {\bf Theorem (Zhou)} \cite[Theorem 3.10]{Z} Let $\mathcal{C}$ be an extriangulated category satisfying Condition (WIC), and $(\U,\widetilde{\V})$ and $(\widetilde{\U},\V)$ be two hereditary cotorsion pairs satisfying $\widetilde{\U}\subseteq \U$, $\widetilde{\V}\subseteq \V$ and $\U\bigcap\widetilde{\V}= \widetilde{\U}\bigcap\V$. Then there is a subcategory $\mathcal{W}$ such that $(\U, \mathcal{W}, \V)$ forms a Hovey triple.

\vspace{2mm}

Recently, Di, Li, and Liang generalized the aforementioned theorem to more general categories by introducing weak factorization systems as counterparts to complete cotorsion pairs (see \cite{DLLiang}). Specifically, for two compatible weak factorization systems $(\U,\widetilde{\V})$ and $(\widetilde{\U},\V)$ satisfying certain properties, they provided a method to construct model structures on general categories with pushouts along morphisms in $\U$ and pullbacks along morphisms in $\V$ (see \cite[Theorem C]{DLLiang}). They also showed that, in the case where $\mathcal{C}$ is an abelian category, a pair $(\U,\V)$ of classes of objects in $\mathcal{C}$ forms a complete and hereditary cotorsion pair if and only if the pair $(\mathrm{Mon(}\U{)},\mathrm{Epi(}\V{)})$ is a weak factorization system under additional assumptions (see \cite[Theorem A]{DLLiang}), where
$$\mathrm{Mon}(\U)=\{f\: \mid f ~\text{is a monomorphism with~} \mathrm{Coker}(f)\in\U\}$$
and
$$\mathrm{Epi}(\V)=\{f\:\mid f ~\text{is an epimorphism with~}\mathrm{Ker}(f)\in\V\}.$$
Moreover, two cotorsion pairs are compatible if and only if the corresponding weak factorization systems are compatible, see \cite[Theorem B]{DLLiang}.

The results above establish a connection between weak factorization systems and complete cotorsion pairs. The aim of this paper is to generalize these results from abelian categories to weakly idempotent complete extriangulated categories.

The paper is organized as follows. In Section 2, we recall the definition of an extriangulated category and outline some basic properties that will be used later. In Section 3, we introduce the concept of admissible weak factorization systems and establish a bijection between cotorsion pairs and admissible weak factorization systems (see Theorem \ref{Thm1}). As a consequence, we provide equivalences of hereditary cotorsion pairs via admissible weak factorization systems under certain conditions (see Theorem \ref{Thm2}). Finally, we give the equivalences of compatible cotorsion pairs through compatible admissible weak factorization systems (see Theorem \ref{Thm3}).

\section{\bf Preliminaries}
Throughout this paper, $\C$ denotes an additive category. By the term $``subcategory"$ we always mean a full additive subcategory of an additive category closed under isomorphisms and direct summands.
 We denote by ${\mathcal{\C}}(A, B)$ the set of morphisms from $A$ to $B$ in $\C$.

Let us briefly recall some definitions and basic properties of extriangulated categories from \cite{NP}. We omit some details here, but the reader can find them in \cite{NP}.

Assume that $\mathbb{E}: \mathcal{C}^{\rm op}\times \mathcal{C}\rightarrow {\rm Ab}$ is an additive bifunctor, where $\mathcal{C}$ is an additive category and ${\rm Ab}$ is the category of abelian groups. For any objects $A, C\in\mathcal{C}$, an element $\delta\in \mathbb{E}(C,A)$ is called an $\mathbb{E}$-extension.
Then formally, an $\mathbb{E}$-extension is a triple $(A,\delta,C)$.
For any $A$ and $C$, the zero element $0\in \E(C,A)$ is called the split $\E$-extension.

Let $\mathfrak{s}$ be a correspondence which associates an equivalence class $\mathfrak{s}(\delta)=\xymatrix@C=0.8cm{[A\ar[r]^x
 &B\ar[r]^y&C]}$ to any $\mathbb{E}$-extension $\delta\in\mathbb{E}(C, A)$. This $\mathfrak{s}$ is called a {\it realization} of $\mathbb{E}$, if it makes the diagram in \cite[Definition 2.9]{NP} commutative.
 A triplet $(\mathcal{C}, \mathbb{E}, \mathfrak{s})$ is called an {\it extriangulated category} if it satisfies the following conditions.
\begin{enumerate}
\item $\mathbb{E}\colon\mathcal{C}^{\rm op}\times \mathcal{C}\rightarrow \rm{Ab}$ is an additive bifunctor.

\item $\mathfrak{s}$ is an additive realization of $\mathbb{E}$.

\item $\mathbb{E}$ and $\mathfrak{s}$ satisfy certain axioms in \cite[Definition 2.12]{NP}.
\end{enumerate}

In particular, we recall the following axioms which will be used later:

{\rm\bf (ET4)}~ Let $\delta\in\mathbb{E}(D,A)$ and $\delta'\in\mathbb{E}(F, B)$ be $\mathbb{E}$-extensions realized by
 \begin{center} $\xymatrix{A\ar[r]^f&B\ar[r]^{f'}&D}$ and $\xymatrix{B\ar[r]^g&C\ar[r]^{g'}&F}$\end{center}
 respectively. Then there { exist} an object $E\in\mathcal{C}$, a commutative diagram
 $$\xymatrix{A\ar[r]^f\ar@{=}[d]&B\ar[r]^{f'}\ar[d]_g&D\ar[d]^d\\
A\ar[r]^h&C\ar[r]^{h'}\ar[d]_{g'}&E\ar[d]^e\\
&F\ar@{=}[r]&F}$$
in $\mathcal{C}$, and an $\mathbb{E}$-extension $\delta^{''}\in\mathbb{E}(E, A)$ realized by $\xymatrix{A\ar[r]^h&C\ar[r]^{h'}&E,}$
which satisfy the following compatibilities.

\begin{enumerate}
\item[(i)] $\xymatrix{D\ar[r]^d&E\ar[r]^{e}&F}$ realizes $\mathbb{E}(F,f')(\delta')$,

\item[(ii)] $\mathbb{E}(d,A)(\delta^{''})=\delta$,

\item[(iii)] $\mathbb{E}(E,f)(\delta^{''})=\mathbb{E}(e,B)(\delta')$.
\end{enumerate}

{\rm\bf (ET4)$^{\rm op}$}\hspace{2mm} Dual of (ET4).

\begin{rem}
Note that both exact categories and triangulated categories are extriangulated categories $($see \cite[Example 2.13]{NP}$)$ and extension closed subcategories of extriangulated categories are
again extriangulated $($see \cite[Remark 2.18]{NP}$)$. Moreover, there exist extriangulated categories which
are neither exact categories nor triangulated categories $($see \cite[Proposition 3.30]{NP}, \cite[Example 4.14]{ZZ} and \cite[Remark 3.3]{HZZ}$)$.
\end{rem}

%
%
\begin{lem}{\rm \cite[Corollary 3.12]{NP}}  Let $(\mathcal{C}, \mathbb{E}, \mathfrak{s})$ be an extriangulated category and $$\xymatrix@C=2em{A\ar[r]^{x}&B\ar[r]^{y}&C\ar@{-->}[r]^{\delta}&}$$ an $\mathbb{E}$-triangle. Then we have the following long exact sequences:

$\xymatrix@C=1cm{\mathcal{C}(C, -)\ar[r]^{\mathcal{C}(y, -)}&\mathcal{C}(B, -)\ar[r]^{\mathcal{C}(x, -)}&\mathcal{C}(A, -)\ar[r]^{\delta^\sharp}&\mathbb{E}(C, -)\ar[r]^{\mathbb{E}(y, -)}&\mathbb{E}(B, -)\ar[r]^{\mathbb{E}(x, -)}&\mathbb{E}(A, -);}$

$\xymatrix@C=1cm{\mathcal{C}(-, A)\ar[r]^{\mathcal{C}(-, x)}&\mathcal{C}(-, B)\ar[r]^{\mathcal{C}(-, y)}&\mathcal{C}(-, C)\ar[r]^{\delta_\sharp}&\mathbb{E}(-, A)\ar[r]^{\mathbb{E}(-, x)}&\mathbb{E}(-, B)\ar[r]^{\mathbb{E}(-, y)}&\mathbb{E}(-, C),}$

\noindent where natural transformations $\delta_\sharp$ and $\delta^\sharp$ are induced by $\mathbb{E}$-extension $\delta\in \mathbb{E}(C, A)$ via Yoneda's lemma.
\end{lem}

Let $\mathcal{C},\mathbb{E}$ be as above, we use the following  notations:

$\bullet$ A sequence $\xymatrix@C=0.6cm{A \ar[r]^{x} & B \ar[r]^{y} & C }$ is called a \emph{conflation} if it realizes some $\mathbb{E}$-extension $\delta\in\mathbb{E}(C, A)$.
In this case, $x$ is called an {\it inflation}, $y$ is called a {\it deflation}, and we write it as $$\xymatrix{A\ar[r]^x&B\ar[r]^{y}&C\ar@{-->}[r]^{\delta}&.}$$
We usually do not write this ``$\delta$" if it is not used in the argument.

$\bullet$ Given an $\mathbb{E}$-triangle $\xymatrix{A \ar[r]^{x} & B \ar[r]^{y} & C \ar@{-->}[r]^{\delta}&,}$ we call $A$ the \emph{CoCone} of $y:B\rightarrow C$ and denoted by $\mathrm{CoCone}(y)$; meanwhile we call  $C$ the \emph{Cone} of $x:A\rightarrow B$ and denoted by $\mathrm{Cone}(x)$.

\section{\bf Main results}
We begin this section with the following definitions.
\begin{df}
Given two morphisms $l: A\to B$, $r: C\to D$ morphisms in $\mathcal{C}$. We write $l\Box r$ if for any commutative square given by the solid arrows
    $$\xymatrix{
  A \ar[d]_{l} \ar[r]^{f}
                & C \ar[d]^{r}  \\
  B \ar@{-->}[ur]^{t} \ar[r]_{g}
                & D             }$$
a morphism depicted by the diagonal dotted arrow exists such that both the triangles commute.

Given $l,r$ such that  $l\Box r$, we say that $l$ has the \emph{left lifting property} with respect to $r$ and $r$ has the \emph{right lifting property} with respect to $l$.

For a class $\mathcal{M}$ of morphisms in $\C,$ denoted by $\mathcal{M}^{\Box}$ the class of morphisms $r$ in $\C$ having the right lifting property with respect to all morphisms $l\in \mathcal{M}.$
The class ${^{\Box}\mathcal{M}}$ is  defined dually.
\end{df}

\begin{df}\cite{B}
A pair $(\mathcal{L},\mathcal{R})$ of classes of morphisms in $\mathcal{C}$ is called a \emph{weak factorization system} if $\mathcal{L}^{\Box}=\mathcal{R}$ and $^{\Box}\mathcal{R}=\mathcal{L}$, and for every morphism $h$ in $\mathcal{C}$ there is a factorization $h=gf$ with $f\in \mathcal{L}$ and $g\in\mathcal{R}.$
\end{df}

\begin{rmk}\label{rm1}
Let $(\mathcal{L},\mathcal{R})$ be a weak factorization system.
Then the class $\mathcal{L}$ and $\mathcal{R}$ are closed under compositions and retracts, and the isomorphisms by \cite[Proposition D.1.2]{Joyal}.
\end{rmk}

\begin{df}
Let $(\mathcal{C}, \mathbb{E}, \mathfrak{s})$ be an extriangulated category and $(\mathcal{L},\mathcal{R})$ be a weak factorization system.
 We call $(\mathcal{L},\mathcal{R})$ the \emph{admissible weak factorization system} if the following two conditions are satisfied for a morphism $f$ in $\C$:

(1) $f\in \mathcal{L}$ if and only if $f$ is an inflation and $0\ra \Zh(f)$ belongs to $\mathcal{L}.$

(2) $f\in \mathcal{R}$ if and only if $f$ is a deflation and $\YZ(f)\ra 0$ belongs to $\mathcal{R}.$
\end{df}
Given a subcategory $\mathcal{X}$  of $\mathcal{C}$, we put
\begin{center}
$\mathrm{Defl}~\X=\{f\: \mid f ~\text{is a deflation with~} \mathrm{CoCone}(f)\in\X\}$.
\end{center}
Dually
\begin{center}
$\mathrm{Infl}~\X=\{f\:\mid f ~\text{is a inflation with~}\mathrm{Cone}(f)\in\X\}$.
\end{center}

For a subcategory $\mathcal{X}$ of $\mathcal{C}$, define $$\mathcal{X}^{\perp}=\{Y\in\mathcal{C}~|~\E(X, Y)=0~ \mathrm{for~and~all~} X \in \mathcal{X}\}.$$
Similarly, we can define $^{\perp}\mathcal{X}.$

\begin{lem}\label{lemma:WFS1}
Let $(\mathcal{C}, \mathbb{E}, \mathfrak{s})$ be an extriangulated category.
If $\U\subseteq\C$ is a class of objects and $g:X\ra Y$ is a deflation such that
$f\Box g$ for each inflation $f$ with $\mathrm{Cone}(f)\in \U,$
then $\YZ(g)\in \U^{\perp}$.
Dually, if $\V\subseteq\C$ {is a class of objects} and $f:W\ra Z$ is an inflation such that $f\Box g$
for each deflation $g$ with $\YZ(g)\in \V,$ then $\Zh(f)\in {^{\perp}\V}$.
\end{lem}

\begin{proof}
Since $g:X\ra Y$ is a deflation, there is an $\E$-triangle
$\xymatrix{K\ar[r]^{h}&X\ar[r]^{g}&Y\ar@{-->}[r]^{\rho}&}$.
For any $U\in \U$ and take $\delta\in \E(U,K).$
Then there exists a conflation $K\xrightarrow{i} E\xrightarrow{} U$
which realizes $\delta.$
Thus we have the following commutative diagram by \cite[Proposition 3.15]{NP}:
$$\xymatrix{
     K\ar[d]_{i} \ar[r]^{h} & X \ar[d]^{\widetilde{i}}\ar[r]^{g}&Y\ar@{=}[d] \ar@{-->}[r]^\rho& \\
  E \ar[d]_{} \ar[r]^{\widetilde{h}} & W \ar[d]^{} \ar[r]^{\widetilde{g}} & Y\ar@{-->}[r]^\eta& \\
  U \ar@{=}[r] \ar@{-->}[d]_{\delta} & U\ar@{-->}[d]^{\varepsilon}\\
  & &}
$$
Consider the following commutative diagram
$$\xymatrix{
X\ar[d]_{\widetilde{i}}\ar@{=}[r] & X\ar[d]^{g}\\
W\ar[r]^{\widetilde{g}} \ar@{-->}[ur]^{t} & Y.
}
$$
Since $\widetilde{i}$ is a inflation with $\Zh(\widetilde{i})=U\in \U,$
there exists $t:W\ra X$ such that $t \widetilde{i}=1_X$ and  $\widetilde{g}=gt$. By $\rm (ET 3)^{op}$, there is a morphism of $\mathbb{E}$-triangles $$\xymatrix{E\ar[r]^{\widetilde{h}}\ar[d]^j&W\ar[d]^t\ar[r]^{\widetilde{g}}&Y\ar@{=}[d]\ar@{-->}[r]^\eta&\\
K\ar[r]^h&X\ar[r]^{g}&Y\ar@{-->}[r]^{\rho}&.}$$
Therefore, there is a morphism of $\mathbb{E}$-triangles
$$\xymatrix{K\ar[r]^{h}\ar[d]^{ji}&X\ar@{=}[d]\ar[r]^{g}&Y\ar@{=}[d]\ar@{-->}[r]^\rho&\\
K\ar[r]^h&X\ar[r]^{g}&Y\ar@{-->}[r]^{\rho}&,}$$
which implies that $ji$ is an isomorphism. So one can get that $i$ is a  split inflation, and $\delta=0$. Thus $K\in \U^{\perp}.$

The other part of the lemma is dual.
\end{proof}

\begin{df}
Let $(\mathcal{C}, \mathbb{E}, \mathfrak{s})$ be an extriangulated category and $\W$ a class of objects in $\C.$

(1) $\W$ is \emph{closed under cocones of deflations} if for any $\E$-triangle
$\xymatrix{A\ar[r]^{x}&B\ar[r]^{y}&C\ar@{-->}[r]^{\delta}&}$ which satisfies $B,C\in \W$, we have $A\in \W.$

(2) $\W$ is \emph{closed under cones of inflations} if for any $\E$-triangle
$\xymatrix{A\ar[r]^{x}&B\ar[r]^{y}&C\ar@{-->}[r]^{\delta}&}$ which satisfies $A,B\in \W$, we have $C\in \W.$

(3) $\W$ is \emph{closed under extensions} if for any $\E$-triangle
$\xymatrix{A\ar[r]^{x}&B\ar[r]^{y}&C\ar@{-->}[r]^{\delta}&}$ which satisfies $A,C\in \W$, we have $B\in \W.$

(4) $\W$ is called \em{thick} if it is closed under direct summands and satisfies the 2-out-of-3 property:
for any $\E$-triangle $\xymatrix{A\ar[r]^{x}&B\ar[r]^{y}&C\ar@{-->}[r]^{\delta}&}$ in $\C$ with two terms in $\W$, the third term
belongs to $\W$ as well.
\end{df}

\begin{df} \emph{\cite[Definition 4.1]{NP}}\label{df:cotorsion pair}
{\rm Let $(\mathcal{C}, \mathbb{E}, \mathfrak{s})$ be an extriangulated category, and let $\mathcal{U}$, $\mathcal{V}$ $\subseteq$ $\mathcal{C}$ be a pair of full additive subcategories, closed
under isomorphisms and direct summands. The pair ($\mathcal{U}$, $\mathcal{V}$) is called a {\it cotorsion
pair} on $\mathcal{C}$ if it satisfies the following conditions:

(1) $\mathbb{E}(\mathcal{U}, \mathcal{V})=0$;

(2) For any $C \in{\mathcal{C}}$, there exists a conflation $V^{C}\rightarrow U^{C}\rightarrow C$ satisfying
$U^{C}\in{\mathcal{U}}$ and $V^{C}\in{\mathcal{V}}$;

(3) For any $C \in{\mathcal{C}}$, there exists a conflation $C\rightarrow V_{C} \rightarrow U_{C}$ satisfying
$U_{C}\in{\mathcal{U}}$ and $V_{C}\in{\mathcal{V}}$.
}
\end{df}

 A cotorsion pair $(\U,\V)$ is called {\em hereditary} if $\U$ is closed under cocones of deflations and $\V$ is closed under cones of inflations.
In fact, cotorsion pair $(\U,\V)$ is hereditary if and only if $\U$ is closed under cocones of deflations if and only if $\V$ is closed under cones of inflations by \cite[Proposition 2.8]{HZZ1}.

{ \begin{df}\cite[Condition 5.8]{NP}
An extriangulated category $(\mathcal{C}, \mathbb{E}, \mathfrak{s})$
satisfies {\rm Condition (WIC)} if the following conditions hold:

(1) If $f\in\C(A,B)$ and { $g\in\C(B,C)$ }is a pair of composable morphisms in $\C$ and $gf$ is an inflation, then so is $f.$

(2) If $f\in \C(A,B)$ and {$g\in\C(B,C)$ } is a pair of composable morphisms in $\C$ and $gf$ is a deflation, then so is $g.$
\end{df}

\begin{remark}
{\rm (1)} If the category $\C$ is exact, then $(\mathcal{C}, \mathbb{E}, \mathfrak{s})$ satisfies {\rm Condition (WIC)} if and only if $\C$ is weakly idempotent complete.

{\rm (2)} If $\C$ is a triangulated category, then {\rm Condition (WIC)} is automatically satisfied.
\end{remark}

{ It is worth nothing that an extriangulated category $\mathcal{C}$ satisfies Condition (WIC) if and only if $\mathcal{C}$ is weakly idempotent complete, see \cite[Proposition C]{K}}. In the following, we always assume that any extriangulated category is weakly idempotent complete}.

\begin{df}
Given a class $\mathcal{L}$ of inflations, we put
$$\Zh(\mathcal{L}):=\{U\in\mathcal{C}~|~U\cong\Zh(f) ~\text{with}~ f\in \mathcal{L}\}.$$
Dually, for a class $\mathcal{R}$ of deflations,
$$\YZ(\mathcal{R}):=\{V\in\mathcal{C}~|~V\cong \YZ(f)~\text{with}~f\in \mathcal{R}\}.$$
\end{df}

\begin{thm}\label{Thm1}
Let $(\mathcal{C}, \mathbb{E}, \mathfrak{s})$ be an extriangulated category.
Then assignments
\begin{center}
$(\U, \V)\mapsto (\I~\U, \D~\V)$ $~~$ and $~~$ $(\mathcal{L},\mathcal{R})\mapsto ( \Zh(\mathcal{L}), \YZ(\mathcal{R}))$
\end{center}
gives the mutually inverse bijections between the following classes:

{\rm (1)} Cotorsion pairs $(\U,\V)$ on $\C$;

{\rm (2)} Admissible weak factorization systems $(\mathcal{L},\mathcal{R}).$
\end{thm}

\begin{proof}
Assume that $(\U,\V)$ is a cotorsion pair on $\C$. We shall prove that every morphism $h:X\ra Y$ factorizes as $h=gf$ with $f\in \I~\U$ and $g\in \D~\V.$
Suppose first that $h$ is an inflation.
Then there is an $\E$-triangle $\xymatrix{X\ar[r]^{h}&Y\ar[r]^{y}&C\ar@{-->}[r]^{\delta}&}.$
Since $(\U,\V)$ is a cotorsion pair, there is an $\E$-triangle
$\xymatrix{V^{C}\ar[r]^{}&U^{C}\ar[r]^{}&C\ar@{-->}[r]^{}&}$ with $U^C\in\U$ and $V^C\in\V$.
By \cite[Proposition 3.15]{NP}, we obtain a commutative diagram of conflations
$$\scalebox{0.9}[0.9]{     \xymatrix{&V^{C}\ar[d]^{}\ar@{=}[r]&V^{C}\ar[d]_{}\\
X\ar[r]^{f}\ar@{=}[d]&Z\ar[r]^{}\ar[d]_{g}&U^{C}\ar[d]^{}&\\
X\ar[r]^{h}&Y\ar[r]&C.}}$$

\noindent The corresponding factorization of $h$ then appears in the leftmost square.
A dual argument applies if $h$ is a deflation.
If $h$ is an arbitrary morphism, we can factor it as
$$X\xrightarrow{\tiny\begin{pmatrix}1_X\\0\end{pmatrix}}X\oplus Y\xrightarrow{\tiny\begin{pmatrix}h&1_Y\end{pmatrix}}Y.$$
Note that $X\xrightarrow{\tiny\begin{pmatrix}1_X\\0\end{pmatrix}}X\oplus Y$ is a split inflation.
It follows that we can factor it as $\tiny\begin{pmatrix}1_X\\0\end{pmatrix}=g_{1}f_{1}$
\begin{align*}
\xymatrix{X\ar[dr]_{f_{1}}\ar[rr]^{\tiny\begin{pmatrix}1_X\\0\end{pmatrix}}& &X\oplus Y\\
&W\ar[ur]_{g_{1}}}
\end{align*}
with $f_{1}\in \I~\U$ and
$g_{1}\in \D~\V$.
As $(h,1_{Y})$ is a deflation by \cite[Corollary 3.16]{NP}, $(h,1_{Y})\circ g_{1}$ is a deflation.
Hence we can factor it as $g_{2}f_{2}$ with $f_{2}\in\I~\U$ and $g_{2}\in \D~\V$ by above argument.
It follows that $h=(h, 1_Y){\tiny\begin{pmatrix}1_X\\0\end{pmatrix}}=(h, 1_Y)g_1f_1=g_{2}(f_{2}f_{1}).$
Note that $f_{2} f_{1}$ is an inflation and there is a conflation
$\Zh(f_{1})\rightarrow \Zh(f_{2}f_{1})\rightarrow \Zh(f_{2})$ by ${\rm (\mathrm{ET4})^{op}}$; also see \cite[Lemma 2.11]{HZZ1}.
Since $\U$ is closed under extensions,  it follows that $f_{2}f_{1}\in \I~\U$.
  Next, we shall prove that $\I~\U={^{\Box}\D~\V}$.
If $f\in \I~\U$ and $g\in \D~\V$,
then $f\Box g$ by \cite[Lemma 2.13]{HZZ1}, which implies $\I~\U\subseteq{^{\Box}\D~\V}$.  If $h\in{^{\Box}\D~\V}$, then $h\Box g$ for any $g\in\D~\V$. By above argument, we have $h=g_1f_1$ with  $f_{1}\in\I~\U$ and $g_{1}\in \D~\V$. Since $h\Box g_1$, there exists a morphism $t: B\rightarrow C$ which makes the following diagram commutative

$$\xymatrix{
A\ar[d]_{h}\ar[r]^{f_1} & C\ar[d]^{g_1}\\
B\ar@{=}[r]\ar@{-->}[ur]^{t} & B.
}
$$
 Since $f_1=th$ is an inflation, one can get that $h$ is an inflation by Condition (WIC). It follows from Lemma \ref{lemma:WFS1} that ${\rm Cone}(h)\in {^\perp{\V}}$, hence $h\in \I~\U$ as $\U={^\perp\V}$, and ${^{\Box}\D~\V}\subseteq\I~\U$. So we have $\I~\U={^{\Box}\D~\V}$, we can prove $\D~\V=\I~\U^\Box$ by a dual argument. Therefore,  $(\I~\U,\D~\V)$ is a weak factorization system. It is easy to check that $(\I~\U,\D~\V)$ is an admissible weak factorization system.

Conversely, let $(\mathcal{L},\mathcal{R})$ be an admissible weak factorization system and put $\U=\Zh(\mathcal{L})$ and $\V=\YZ(\mathcal{R}).$
Since $\mathcal{L}$ and $\mathcal{R}$ are closed under retracts by Remark \ref{rm1}, both $\U$ and $\V$ are closed under direct summands.
By Lemma \ref{lemma:WFS1}, we have $\E(\U,\V)=0$ for each $\U\in\U$ and $V\in\V.$
In order to prove the existence of approximation sequences, consider $X\in \C$ and the following two factorizations
\begin{align*}
\xymatrix{X\ar[dr]_{i_{X}}\ar[rr]& &0\\
&V_{X}\ar[ur]}
\end{align*}
and
\begin{align*}
\xymatrix{0\ar[dr]_{}\ar[rr]& &X\\
&U^{X}\ar[ur]_{p^{X}}}
\end{align*}
with respect to $(\mathcal{L},\mathcal{R}).$
Thus there are conflations $X\ra V_{X}\ra U_{X}$ and $V^{X}\ra U^{X}\ra X$
with $V_{X}, V^{X}\in \V$ and $U_{X}, U^{X}\in \U.$
Hence $(\U,\V)$ is a cotorsion pair on $\C$.
\end{proof}

\rm {Recall from \cite[Definition C.0.20]{Joyal} that a class $\mathcal{L}$ of morphisms in a category $\C$ satisfies {\em the left cancellation property} if $\beta\alpha\in \mathcal{L}$ and $\beta\in \mathcal{L}$ imply $\alpha\in \mathcal{L}.$
Dually, a class $\mathcal{R}$ morphism in $\C$ satisfies {\em the right cancellation property} if $\beta\alpha\in \mathcal{R}$ and $\alpha\in \mathcal{R}$ imply $\beta\in \mathcal{R}.$}

\begin{lem}\label{lem:CoCone}
Let $\X$ be a class of objects in $\C$.
Then the following are equivalent:

{\rm (1)} $\X$ is closed under cocones of deflations.

{\rm (2)} $\I~\X$ satisfies the left cancellation property.
\end{lem}

\begin{proof}
$(1)\Rightarrow (2)$ Let $g:B\ra C$ and $f:A\ra B$ be morphisms in $\C$
such that $gf$ and $g$ are in $\I (\X)$.
Since $\C$ satisfies  Condition (WIC), $f$ is an inflation.
By { \cite[Lemma 3.14]{NP}},  we have the following commutative diagram
\begin{center}
\scalebox{0.9}[0.98]{     \xymatrix{A\ar[r]^{f}\ar@{=}[d]&B\ar[r]^{}\ar[d]_{g}&\Zh(f)\ar[d]^{}&\\
A\ar[r]^{gf}&C\ar[r]^{}\ar[d]_{}&\Zh(gf)\ar[d]^{}&\\
&\Zh(g)\ar@{=}[r]&\Zh(g).}}
\end{center}
Since $\X$ is closed under cocones of deflations by the assumption, and both $\Zh(gf)$ and $\Zh(g)$ are in $\X,$ it follows that $\Zh(f)\in \X$ as well,  so $f\in \I~\X.$

$(2)\Rightarrow(1)$ Let $\xymatrix{X_{1}\ar[r]^{h}&X_{2}\ar[r]^{}&X_{3}\ar@{-->}[r]^{}&}$
be {an $\mathbb{E}$-triangle}  with $X_{2}$ and $X_{3}\in \X$.
Let $l$ be the zero morphism from $0$ to $X_{1}$.
Then $hl\in \I~(\X)$, and so is $l$ as $h$ is in $\I~(\X).$
Thus $X_{1}\in \X.$
\end{proof}

\begin{lem}\label{lem:Co}
Let $\mathcal{Y}$ be a class of objects in $\C$.
Then the following are equivalent:

{\rm (1)} $\mathcal{Y}$ is closed under cones of deflations;

{\rm (2)} $\D~\mathcal{Y}$ satisfies the right cancellation property.
\end{lem}
\begin{proof}
The proof is dual to that of Lemma \ref{lem:CoCone}.
\end{proof}

\begin{lem}\label{lem:LR}
Let $(\U,\V)$ be a pair of class of objects in $\C$ such that $(\I~\U,\D~\V)$  is an admissible weak factorization system.
Then $\I~\U$ satisfies the left cancellation property if and only if $\D~\V$ satisfies the right cancellation property.
\end{lem}
\begin{proof}
We only prove the ``only if~" part as the ``if~" part is a dual statement.
Since $\I~\U$ satisfies the left cancellation property, $\U$ is closed under cocones of deflation by Lemma \ref{lem:CoCone}.
Since $(\U,\V)$ is a cotorsion pair by Theorem \ref{Thm1}, it follows from \cite[Proposition 2.18]{HZZ1} that $\V$ is closed under cones of inflations.
Thus $\D~\V$ satisfies the right cancellation property by Lemma \ref{lem:Co}.
\end{proof}

\begin{thm}\label{Thm2}
Let $(\U,\V)$ be a pair of class of objects in $\C$.
Then the following are equivalent:

{\rm (1)} $(\U,\V)$ is a hereditary cotorsion pair on $\C$;

{\rm (2)} $(\I~\U,\D~\V)$  is an admissible weak factorization system such that $\I~ \U$ satisfies left cancellation property;

{\rm (3)} $(\I~\U,\D~\V)$  is an admissible weak factorization system such that $\D ~\V$ satisfies right cancellation property.
\end{thm}

\begin{proof}
The equivalence follows from Theorem \ref{Thm1} and Lemmas \ref{lem:CoCone},
\ref{lem:Co} and \ref{lem:LR}.
\end{proof}
{A direct consequence of Theorem \ref{Thm2} yields the following corollary.}
\begin{cor}\cite[Theorem A]{DLLiang}
Let $(\U,\V)$ be a pair of class of objects in an abelian category $\mathcal{A}$.
Then the following are equivalent:

{\rm (1)} $(\U,\V)$ is a complete and hereditary cotorsion pair on $\mathcal{A}$;

{\rm (2)} $(\mathrm{Mon}~(\U),\mathrm{Epi}~(\V))$  is a weak factorization system such that $\mathrm{Mon}~(\U)$ satisfies left cancellation property and/or
$\mathrm{Epi}~(\V)$ has the right cancellation property, where
$$\mathrm{Mon}(\U)=\{f\: \mid f ~\text{is a monomorphism with~} \mathrm{Coker}(f)\in\U\}$$
and
$$\mathrm{Epi}(\V)=\{f\: \mid f ~\text{is an epimorphism with~}\mathrm{Ker}(f)\in\V\}.$$
\end{cor}

{We recall the definition of co-t-structures, which was independently introduced by Bondarko \cite{Bondarko} and Pauksztello \cite{Pa}.}
\begin{definition} \cite{Bondarko,MSSS2}  {\rm Let $\mathcal{T}$ be a triangulated category. A \emph{co-t-structure} on $\mathcal{T}$ is a pair $(\mathcal{U},\mathcal{V})$ of subcategories of $\mathcal{T}$ such that
\begin{enumerate}
\item $\mathcal{U}[-1]\subseteq \mathcal{U}$ and $\mathcal{V}[1]\subseteq \mathcal{V}$.

\item $\mathrm{Hom}_{\mathcal{T}}(\mathcal{U}[-1],\mathcal{V})=0$.

\item Any object $T\in{\mathcal{T}}$ has a distinguished triangle
$$\xymatrix@C=0.6cm{X\ar[r]&T\ar[r]&Y\ar[r]&X[1],}$$
in $\mathcal{T}$ with $X\in{\mathcal{U}[-1]}$ and $Y\in{\mathcal{V}}$.
\end{enumerate}}
\end{definition}

{ Since hereditary cotorsion pairs on triangulated category $\mathcal{T}$ are exactly co-t-structures, we have the following result by Theorem \ref{Thm2}.}
\begin{cor}
Let $(\U,\V)$ be a pair of class of objects in a triangulated category $\mathcal{T}$.
Then the following are equivalent:

{\rm (1)} $(\U,\V)$ is a co-t-structure;

{\rm (2)} $(\I~\U,\D~\V)$  is an admissible weak factorization system such that $\I ~\U$ satisfies left cancellation property;

{\rm (3)} $(\I~\U,\D~\V)$  is an admissible weak factorization system such that $\D ~\V$ satisfies right cancellation property.
\end{cor}

\begin{df}\cite[Definition 1.11]{DLLiang}
Two weak factorization systems $(\mathcal{L},\widetilde{\mathcal{R}})$ and
$(\widetilde{\mathcal{L}},\mathcal{R})$ in $\mathcal{C}$ are called {\it compatible} if the following conditions hold:

$(\mathrm{CP1})$ $\widetilde{\mathcal{L}}\subseteq\mathcal{L}$ (or equivalent, $\widetilde{\mathcal{R}}\subseteq\mathcal{R}$);

$(\mathrm{CP2})$ given composable morphisms $\alpha$ and $\beta$ in  $\mathcal{R}$, if two of three morphisms $\alpha$, $\beta$ and $\beta\alpha$ are in $\widetilde{\mathcal{R}}$, then so is the third one;

$(\mathrm{CP3})$ given $l\in\widetilde{\mathcal{L}}$ and $r\in \mathcal{R}$,
if $rl\in \widetilde{\mathcal{L}}$, then $r\in \widetilde{\mathcal{R}}.$

Let $(\mathcal{L},\widetilde{\mathcal{R}})$ and
$(\widetilde{\mathcal{L}},\mathcal{R})$ be two compatible weak factorization systems in $\mathcal{C}$.
We define
$$\W_{\widetilde{\mathcal{L}},\widetilde{\mathcal{R}}}=\{\alpha~|~\alpha~ \text{can be decomposed as }\alpha=\widetilde{r} \widetilde{l}~\text{with}~ \widetilde{l}\in\widetilde{\mathcal{L}}~\text{and}~\widetilde{r}\in \widetilde{\mathcal{R}}\}.$$
\end{df}

\begin{lem}\label{lem:co}
Let $(\mathcal{L},\widetilde{\mathcal{R}})$ and
$(\widetilde{\mathcal{L}},\mathcal{R})$ be two compatible admissible weak factorization systems in extriangulated category $(\mathcal{C}, \mathbb{E}, \mathfrak{s})$. If the composition $g=f h$ is in $\W_{\widetilde{\mathcal{L}},\widetilde{\mathcal{R}}}$ with
$f\in \widetilde{\mathcal{R}}$, then $h\in \W_{\widetilde{\mathcal{L}},\widetilde{\mathcal{R}}}.$
\end{lem}
\begin{proof}
Since $(\mathcal{L},\widetilde{\mathcal{R}})$ is an admissible weak factorization system, we have $h=h_{2} h_{1}$ with $h_{1}\in \mathcal{L}$ and $h_{2}\in \widetilde{\mathcal{R}}.$
Then $g=fh_{2}h_{1}$ with $fh_{2}\in \widetilde{\mathcal{R}}$ by Remark \ref{rm1}.
If the conclusion holds for $h_1$, that is, $h_1\in \W_{\widetilde{\mathcal{L}},\widetilde{\mathcal{R}}}$, then $h_1\in \mathcal{L}\bigcap \W_{\widetilde{\mathcal{L}},\widetilde{\mathcal{R}}}$. It follows from \cite[Lemma 1.12]{DLLiang} that $h_1\in\widetilde{\mathcal{L}}$, so $h=h_2h_1\in \W_{\widetilde{\mathcal{L}},\widetilde{\mathcal{R}}}$.  So we only deal with the special case that $h\in \mathcal{L}$.
Let $h:X\ra Z$ and $f:Z\ra Y$.
Since $g=f h\in \We$, by definition, there is a morphism $\widetilde{l}:X\ra Z'$ in $\widetilde{\mathcal{L}}$ and a morphism $\widetilde{r}:Z'\ra Y$ in $\widetilde{\mathcal{R}}$ such that the following diagram commutes:
$$\xymatrix{X\ar[r]^{h} \ar[d]_{\widetilde{l}} &Z\ar[d]^{f}\\
Z'\ar[r]^{\widetilde{r}} &Y
}$$
 Since $(\mathcal{L},\widetilde{\mathcal{R}})$
 is an admissible weak factorization system,  $f$ and $\widetilde{r}$ are deflations by definition. Thus we have the following commutative diagram:
$$\scalebox{0.92}[1.0]{     \xymatrix{&\YZ(f)\ar[d]^{}\ar@{=}[r]&\YZ(f)\ar[d]_{}\\
\YZ(\widetilde{r})\ar[r]^{}\ar@{=}[d]&P\ar[r]^{\widetilde{f_{2}}}\ar[d]_{\widetilde{f_{1}}}&Z\ar[d]^{f}&\\
\YZ(\widetilde{r})\ar[r]^{}&Z'\ar[r]^{\widetilde{r}}&Y}}$$
Then one has that $\widetilde{f_{1}}$ and $\widetilde{f_{2}}$ are in $\widetilde{\mathcal{R}}$.
Note that
$$\xymatrix{P\ar[r]^{\widetilde{f_{2}}} \ar[d]_{\widetilde{f_{1}}} &Z\ar[d]^{f}\\
Z'\ar[r]^{\widetilde{r}} &Y
}$$
is a weak pullback by the dual of \cite[Lemma 3.13]{NP}.
It follows that there is a morphism $\alpha:X\ra P$ such that $\widetilde{f_{1}}\alpha=\widetilde{l}$ and $\widetilde{f_{2}}\alpha=h$.
Since $(\widetilde{\mathcal{L}},\mathcal{R})$ is an admissible weak factorization system, there is a morphism $\widetilde{l'}:X\ra P'$ in $\widetilde{\mathcal{L}}$ and a morphism $f':P'\ra P$ in $\mathcal{R}$
\begin{align*}
\xymatrix{X\ar[dr]_{\widetilde{l'}}\ar[rr]^{\alpha}& &P\\
&P'\ar[ur]_{f'}}
\end{align*}
with $\alpha=f' \widetilde{l'}.$
Since $\widetilde{l}=\widetilde{f_{1}}\alpha=\widetilde{f_{1}}f'\widetilde{l'}$ in $\widetilde{\mathcal{L}},$ we have that $\widetilde{f_{1}}f'$ belongs to $\widetilde{\mathcal{R}}$ by condition (CP3).
It follows from condition $(\mathrm{CP2})$ that $f'\in \widetilde{\mathcal{R}}.$
Consequently, $h=\widetilde{f_{2}}f'\widetilde{l'}$ is in $\W_{\widetilde{\mathcal{L}},\widetilde{\mathcal{R}}}$.
\end{proof}

\begin{df}
Two cotorsion pairs $(\U,\widetilde{\V})$ and $(\widetilde{\U},\V)$ are {\it compatible} if $\widetilde{\U}\subseteq\U$ (or equivalently $\widetilde{\V}\subseteq\V)$ and $\U\bigcap\widetilde{\V}=\widetilde{\U}\bigcap\V$.
\end{df}

\begin{thm}\label{Thm3}
Let $(\mathcal{C}, \mathbb{E}, \mathfrak{s})$ be an extriangulated category, and let $(\U,\widetilde{\V})$ and $(\widetilde{\U},\V)$ be two pairs of classes of objects in $\C.$
Then the following are equivalent:

{\rm (1)} $(\U,\widetilde{\V})$ and $(\widetilde{\U},\V)$ are compatible cotorsion pairs on $\C$.

{\rm (2)} $(\I~\U,\D~\widetilde{\V})$ and $(\I~ \widetilde{\U},\D~\V)$ are compatible  admissible weak factorization systems.
\end{thm}
\begin{proof}
$(1)\Rightarrow (2)$ By Theorem \ref{Thm1}, $(\I~\U,\D~\widetilde{\V})$ and $(\I~\widetilde{\U},\D~\V)$ are two admissible weak factorization systems.
Now we show that they are compatible. By \cite[Theorem  3.9]{Z}
there is a thick class $\W$ of $\C$ such that $\widetilde{\U}=\U\bigcap\W$ and $\widetilde{\V}=\V\bigcap\W$.
The (CP1) holds clearly.
Suppose that $f:X\ra Y$ and $g:Y\ra Z$ are in $\D~(\V).$
Then $g f$ is a deflation and then there exists a conflation
$$\YZ(f)\ra \YZ(gf)\ra \YZ(g)$$
in $\C$ by \cite[Lemma 2.11]{HZZ1}.
Thus the condition (CP2) holds as $\W$ is a thick class.

To prove (CP3), let $h:A\ra B$ be in $\I~\widetilde{\U}$ and $f:B\ra C$ be in $\D~\V$ such that $g=f h$ is in $\I~\widetilde{\U}.$
Then $h$ and $g$ are inflations with $\Zh(h)$ and $\Zh(g)$ in $\widetilde{\U}$.
Thus by \cite[Lemma 2.12]{HZZ1}, there is a conflation
$$\YZ(f)\ra \Zh(h)\ra \Zh(g).$$
 Since $\W$ is a thick class, it follows that $\YZ(f)$ belongs to $\W$.
Thus $\YZ(f)$ is in $\widetilde{\V}.$

$(2)\Rightarrow(1)$ By Theorem \ref{Thm1}, $(\U,\widetilde{\V})$ and $(\widetilde{\U},\V)$ are two cotorsion pairs in $\C.$
If $X\in \widetilde{\U}$, then $0\ra X$ is in $\I~\widetilde{\U}$
and hence in $\I~\U$, so $X\in \U$ which yields that $\widetilde{\U}\subseteq\U.$
 Next we claim that $\widetilde{\U}\bigcap\V\subseteq\U\bigcap \widetilde{\V}.$
Take $M\in \widetilde{\U}\bigcap\V$.
Then $0\ra M$ is in $\I~\widetilde{\U}$ and $M\ra 0$ is in $\D~\V.$ Since $(\I~\U,\D~\widetilde{\V})$ and $(\I~\widetilde{\U},\D~\V)$ are compatible weak factorization systems, it follows  from the Condition (CP3) that $M\ra 0$ is in $\D~\widetilde{\V}$ as $0\ra 0$ is in $\I~\widetilde{\U}.$
Hence $M\in \widetilde{\V},$ so $M\in \U\bigcap\widetilde{\V}.$

We put $\widetilde{\mathcal{L}}=\I~\widetilde{\U}$ and $\widetilde{\mathcal{R}}=\D~~\widetilde{\V}$.
To prove $\U\bigcap\widetilde{\V}\subseteq\widetilde{\U}\bigcap\V,$
let $N$ be an object in $\U\bigcap\widetilde{\V}.$
Then $0\ra N$ is in $\I~\U$ and $N\ra 0$ is in $\D~\widetilde{\V}$.
Since $0\ra 0$ is in $\We$, it follows from  Lemma \ref{lem:co} that $0\ra N$ is in $\We \bigcap\I~\U$. Thus, by \cite[Lemma 1.12]{DLLiang}, $0\ra N$ is in $\I~\widetilde{\U}$.
So $N\in \widetilde{\U}.$
On the other hand,  $N\in \widetilde{\V}=\U^{\perp}\subseteq\widetilde{\U}^{\perp}=\V.$
Therefore $\U\bigcap\widetilde{\V}\subseteq\widetilde{\U}\bigcap \V.$
We have proven that $\U\bigcap\widetilde{\V}=\widetilde{\U}\bigcap \V$, so $(\widetilde{\U},\V)$ and $(\U,\widetilde{\V})$ are compatible cotorsion pairs.
\end{proof}

\begin{cor}\cite[Theorem B]{DLLiang}
Let $\mathcal{A}$ be an abelian category, and let $(\U,\widetilde{\V})$ and $(\widetilde{\U},\V)$ be two pairs of classes of objects in $\mathcal{A}.$
Then the following are equivalent:

{\rm (1)} $(\U,\widetilde{\V})$ and $(\widetilde{\U},\V)$ are compatible cotorsion pairs.

{\rm (2)} $(\I~\U,\D~\widetilde{\V})$ and $(\I~ \widetilde{\U},\D~\V)$ are compatible weak factorization systems.
\end{cor}

%
%
%

\begin{cor}
Let $\mathcal{T}$ be a triangulated category, and let $(\U,\widetilde{\V})$ and $(\widetilde{\U},\V)$ be two pairs of classes of objects in $\mathcal{T}.$
Then the following are equivalent:

{\rm (1)} $(\U,\widetilde{\V})$ and $(\widetilde{\U},\V)$ are compatible cotorsion pairs on $\mathcal{T}$.

{\rm (2)} $(\I~\U,\D~\widetilde{\V})$ and $(\I~ \widetilde{\U},\D~\V)$ are compatible weak factorization systems.
\end{cor}

\renewcommand\refname
\vspace{4mm}
\textbf{Yajun Ma}\\
School of Mathematics and Physics, Lanzhou Jiaotong University, Lanzhou 730070, P. R. China;\\
E-mail: \textsf{13919042158@163.com}\\[1mm]
\textbf{Hanyang You}\\
School of Mathematics, Hangzhou Normal University,
 Hangzhou 311121, P. R. China;\\
E-mail: youhanyang910328@163.com\\[1mm]
\textbf{Dongdong Zhang}\\
Department of Mathematics, Zhejiang Normal University,
Jihhua 321004, P. R. China;\\
E-mail: \textsf{zdd@zjnu.cn}\\[1mm]
\textbf{Panyue Zhou}\\
School of Mathematics and Statistics, Changsha University of Science and Technology,
Changsha 410114, P. R. China.\\
E-mail: \textsf{panyuezhou@163.com}\\[1mm]
\end{document}